\documentclass[reqno]{amsart}
\usepackage[margin = 1.4in]{geometry}
\usepackage{amsmath, amssymb, amsthm, fancyhdr, verbatim, graphicx}
\usepackage{enumerate}
\usepackage{caption}
\usepackage[all]{xy}
\usepackage{MnSymbol}
\usepackage[dvipsnames]{xcolor}
\usepackage{mathrsfs}
\usepackage{tikz-cd}
\usepackage{framed}
\usepackage[hidelinks, pagebackref]{hyperref}
\hypersetup{colorlinks=true, linkcolor=Violet,citecolor=Emerald}
\usepackage[titletoc]{appendix}
\usepackage{bbm}
\usepackage{lipsum}
\usepackage{adjustbox}
\usepackage{mathfyz}
\usepackage{mathzyw}
\usepackage{envi}
\usepackage{stmaryrd}
\usepackage{leftindex}
\usepackage[english]{babel}
\usepackage{MnSymbol}
\numberwithin{equation}{section}

\makeatletter
\def\th@remark{%
  \thm@headfont{\bfseries}%
  \normalfont
  \thm@preskip \thm@preskip 
  \thm@postskip\thm@preskip
}
\def\imod#1{\allowbreak\mkern5mu({\operator@font mod}\,\,#1)}
\makeatother

\title[Characteristic $2$ Symplectic Commuting pairs]{Commuting pairs in symplectic Lie algebras over finite fields of characteristic two}

\author{Liam May}
\address{Department of Mathematics, Massachusetts Institute of Technology, 77 Massachusetts Avenue, Cambridge, MA 02139, USA}
\email{liammay2@mit.edu}

\author{Yong Yang}
\address{Department of Mathematics, MCS 470, 601 University Drive, San Marcos, TX 78666-4684}
\email{yang@txstate.edu}

\begin{document}

\begin{abstract}
    We find a generating function for the number of ordered commuting pairs of elements in a symplectic Lie algebra over a finite field of characteristic $2$, following the analogous work done in odd characteristic in \cite{FulmanGuralnick2016}.
\end{abstract}

\maketitle
\tableofcontents

\section{Introduction}

Let $q$ be a fixed power of $2$, and $V$ a $2n$-dimensional symplectic vector space over $k = \F_q$, equipped with symplectic form $B$. Then $\sp(V)$ is the associated Lie algebra. Let $N_n(q)$ denote the number of ordered pairs of commuting elements of the associated Lie algebra, that is \[
    N_n = \#\{(X,Y) \in \sp(V) \times \sp(V) \mid [X,Y] = 0\}.
\]
The goal of this paper is to compute a generating function for the normalized quantities \[
    \overline N_n(q) = \frac{N_n(q)}{|\Sp_{2n}(q)|}.
\] 
Fulman and Guralnick computed generating functions for the analogous quantities regarding $\mathfrak{gl}_n(\F_q)$, finite unitary Lie algebras, and for $\mathfrak{sp}_{2n}(\F_q)$ in odd characteristic \cite{FulmanGuralnick2016}. They do not treat the characteristic $2$ symplectic case, and note that there are several difficulties.

\section{Preliminaries}

\subsection{Basic Definitions}

Let $q$ be a fixed power of $2$, $k = \F_q$, and we fix a \emph{symplectic vector} space $V \cong \F_q^{2n}$. That is, $V$ is equipped with a \emph{symplectic form} $B: V\times V \to k$ which is \emph{alternating}, meaning $B(v,v) = 0$, and \emph{nondegenerate}, meaning the map $v\mapsto B(-, v)$ is an isomorphism of $V\to V^*$.
Then the subgroup of $\GL(V)$ which respects the symplectic form $B$, \[
   \Sp(V) := \{g\in \GL \mid B(gv, gw) = B(v,w) \;\forall v,w\in V\},
\] is the \emph{symplectic group} of $V$ over $k$. The \emph{Lie algebra} $\sp(V)$ of $\Sp(V)$ may be written as \[
    \sp(V) := \{X\in \End(V) \mid B(Xv, w) = B(v, Xw)\;\forall v,w\in V\}.
\] For any two $2n$-dimensional symplectic vector spaces $W'$ and $W''$ over $k$, equipped with symplectic forms $B'$ and $B''$, we say that an isomorphism $\phi:W'\to W''$ is an \emph{isometry} if $\phi$ moreover respects the forms. It is a fact that any two symplectic vector spaces of the same dimension are isometric, thereby inducing an isomorphism of symplectic Lie algebras. In light of this, it is reasonable to use the notation $\Sp(2n, q)$ and $\sp(2n, q)$ without reference to $V$.

\medskip
\subsection{Centralizer Presentation}
The object of this paper is the set of ordered commuting pairs $(X,Y)$ in $\sp(V)^2$. Fixing the first entry $X$, it is clear that the admissible second entries $Y$ are, by definition, the elements of the Lie algebra centralizer $C_{\sp(V)}(X)$. Therefore we can express the quantity $N_n(q)$ as \[
    N_n(q) = \sum_{X\in \sp(V)} | C_{\sp(V)}(X)|.
\] 
Further, $\Sp(V)$ has a well defined action on $\sp(V)$ by conjugation, and \[
    C_{\sp(V)}(X) \cong C_{\sp(V)}(gXg^{-1}).
\]
Since $|C_{\sp(V)}(-)|$ is therefore constant on $\Sp(V)$-orbits, we may write \[
    N_n(q) = \sum_{[X]} |[X]|\cdot |C_{\sp(V)}(X)|,
\]
where the sum is taken over $\Sp(V)$-orbits $[X]$ in $\sp(V)$, and $X$ is a representative of $[X]$. By the Orbit-Stabilizer Theorem, \[
    |[X]| = \frac{|\Sp(V)|}{|C_{\Sp(V)}(X)|}, 
\]
so we may write \[
    N_n(q) = \sum_{[X]} \frac{|\Sp(V)|}{|C_{\Sp(V)}(X)|} \cdot |C_{\sp(V)}(X)|. 
\]
Rewriting as \[
    \overline N_n(q) = \sum_{[X]} \frac{|C_{\sp(V)}(X)|}{|C_{\Sp(V)}(X)|}, \tag{1}\label{tag:1}
\]
it is clear why we consider the normalized $\overline N_n(q)$. We have reduced to the problem of understanding the group centralizer and Lie algebra centralizer of $X$.

\section{A Symplectic Orthogonal Decomposition}

\subsection{Overview} 

The purpose of this section and the next is to provide a decomposition of the centralizers in the numerator and denominator of summands in $(\ref{tag:1})$. In this section, we take the primary decomposition of $V$ determined by $X$, and show that indeed this decomposition is orthogonal with respect to the symplectic form on $V$, where each orthogonal summand can be viewed as a symplectic vector space of its own, defined over an extension of $k$. We also draw attention to the nilpotent component of the restriction of $X$ to each primary component, showing it to be a symplectic element. In the next section, we demonstrate the $\Sp(V)$-orbits in $\sp(V)$, over which the summation $(\ref{tag:1})$ is taken, correspond to nilpotent orbits in the symplectic Lie algebras of each primary summand.

\subsection{The Decomposition}

For a given $X\in \End(V)$, we may view $V$ as a $k[t]$-module by letting $t$ act on $V$ as $X$. By the Structure Theorem for a finite generated module over a PID, there is a decomposition\[
    V \cong \bigoplus_{f} V_f(X)
\]
ranging over monic irreducible $f\in k[t]$, called the \emph{primary decomposition} of $V$ (with respect to $X$). Here, the summands are \[
    V_f(X) = \{v\in V\mid f(X)^r v = 0,\;\; r>> 0\}.
\]
This is well-defined, since the nested chain of kernels of $f(X)^r$ stabilizes. Notice that for any nontrivial summand $V_f(X)$ of the primary decomposition, $f$ divides the minimal polynomial $F$ of $X$. For suppose that $f(X)^r$ is not injective, but $f^r$ and $F$ are coprime, say $af^r + bF = 1$. Then $a(X)f(X)^r + b(X) F(X) = a(X)f(X)^r$ is not injective, a contradiction. The critical structural fact that we can observe is that, when $X\in \sp(V)$, this decomposition is orthogonal and nondegenerate with respect to the symplectic structure on $V$. We say that two subspaces $W', W''\subseteq V$ are \emph{orthogonal}, written $W'\perp W''$, if $B(w', w'') = 0$ for all $w'\in W'$ and $w''\in W''$. We say that a subspace $W\subseteq V$ is \emph{nondegenerate} if $B$ restricted to $W$ is nondegenerate.

\begin{lemma}\label{lemma:orthog}
    Let $X\in \sp(V)$, and let $V = \bigoplus_f V_f(X)$ be the associated primary decomposition. Then for all $f\ne g$, $V_f(X) \perp V_g(X)$, and moreover each $V_f(X)$ is nondegenerate.
\end{lemma}

\begin{proof}
    To show orthogonality, take nonzero elements $v\in V_f(X)$ and $w\in V_g(X)$ for distinct monic irreducibles $f$ and $g$. By assumption, $f(X)^r v = 0$ for all $r$ sufficiently large. Additionally, for all $r>0$, $f(X)^r$ is injective on $V_g(X)$, hence an isomorphism. So we may write $w = f(X)^r u$ for some $u\in V_g(X)$. Then \[
        B(v,w) = B(v, f(X)^r u) = B(f(X)^rv, u) = B(0,u) = 0.
    \]
    We have used the fact that since $X$ is self-adjoint for $B$ by assumption, so is $f(X)^r$.

    It remains to show that $B$ is nondegenerate on each $V_f$. Suppose $v\in V_f(X)$ satisfies $B(v,u) = 0$ for all $u\in V_f(X)$. Also $B(v,u) = 0$ for all $u\in V_g$ for $g \ne f$, by orthogonality, and hence for all $u\in V$. Since $B$ is nondegenerate on $V$, $v = 0$. Thus $B$ is nondegenerate on $V_f$. 
\end{proof}

Now fix a monic irreducible polynomial $f\in k[t]$ of degree $d$, and consider $X_f = X|_{V_f(X)}$. Assume that $V_f(X)$ is a nonzero component of the primary decomposition. This restriction $X_f$ is well-defined as $V_f(X)$ is $X$-stable. It is well-known that $X_f$ admits a unique \emph{Jordan-Chevalley decomposition} $X_f = S_{X,f} + N_{X,f}$ where $S_{X,f}$ is semisimple, $N_{X,f}$ is nilpotent, and $S_{X,f} N_{X,f} = N_{X,f} S_{X,f}$. It follows from \cite[Prop. 1.22]{MilneLAG} that in fact $S_{X,f}$ and $N_{X,f}$ are polynomials in $X_f$. We call $S_{X,f}$ the \emph{semisimple component} of $X_f$ and $N_{X,f}$ the \emph{nilpotent component}.

Set $K_f = k[t]/(f(t)) \cong \F_{q^d}$, and write $K_f = k(\alpha_f)$, where $\alpha_f$ is the image of $t$ in $K_f$ and hence a root of $f$. We claim that $V_f(X)$ has the structure of a $K_f$-vector space, and that the nilpotent part $N_{X,f}$ is $K_f$-linear. To give $V_f(X)$ the structure of a $K_f$ vector space, the idea is to let $\alpha_f$ act as the semisimple part $S_{X,f}$, which we explain below.

\medskip

First note that the minimal polynomial of $S_{X,f}$ divides $f$, therefore $f(S_{X,f}) = 0$. This can be seen as follows. After basechange to $\overline k$, there is a decomposition \[
    V_f(X)' := V_f(X)\otimes_{k}\overline k \cong \bigoplus_\beta E_\beta
\]
into eigenspaces of $S_{X,f}$, as $S_{X,f}$ is semisimple and $k$ is perfect. Since $S_{X,f}$ and $N_{X,f}$ commute, $N_{X,f}$ preserves each $E_\beta$, so the restriction of $X_f$ to each $E_\beta$ takes the form $\beta I + N$ for some nilpotent $N$. Thus every eigenvalue $\beta$ of $S_{X,f}$ is also an eigenvalue of $X_f$. Since we assumed that $f(X_f)^r = 0$ for some $r$, the minimal polynomial of $X_f$ divides $f^r$, and therefore all eigenvalues of $X_f$ are roots of $f$ in $\overline \F_q$. Therefore the minimal polynomial of $S_{X,f}$ divides $f$ as claimed. 

The action of $K_f$ on $V_f(X)$ is given as follows. First consider the map \[
    \F_q[t] \longrightarrow \End_k(V_f)
\]
which sends $t\mapsto S_{X,f}$. Now since $f(S_{X,f}) = 0$, this descends to a map \[
    K_f \longrightarrow \End_k(V_f),
\]
as needed. Finally, we show that the nilpotent component $N_{X,f}$ is $K_f$-linear. For an element $a = p(\alpha_f)$ of $K_f$, where $p$ is a polynomial, $N_{X,f}(a\cdot v) = N_{X,f}(p(S_{X,f}) v) = p(S_{X,f})N_{X,f}(v)$, which shows that $N_{X,f}$ is $K_f$-linear.  

\medskip

Having shown that each $V_f(X)$ is a $K_f$-vector space, we now aim to show that this vector space is moreover symplectic. To define the symplectic form of $V_f(X)$, we use the field trace on $K_f/k$. Recall the \emph{field trace}, denoted $\Tr = \Tr_{K_f / k}$, is the map $\Tr : K_f \to k$ given by $x\mapsto x + x^q + \dots + x^{q^{d-1}}$. The associated \emph{trace pairing} $\langle -,- \rangle : K_f \times K_f \to \F_q$ defined by $\langle x,y \rangle = \Tr(xy)$. We will use that this pairing is perfect, thus, there is a canonical isomorphism \[
    K_f \longrightarrow \Hom_{k}(K_f, k)
\]
given by $x\mapsto (y \mapsto \Tr(xy))$.

\begin{lemma}\label{lemma:2-2}
    There exists a unique $K_f$-bilinear form \[
        B_f : V_f(X) \times V_f(X) \to K_f
    \]
    which satisfies \[
        \Tr(a B_f(v,w)) = B(av,w),
    \]
    for all $a\in K_f$, and this form is moreover alternating and nondegenerate.
\end{lemma}

\begin{proof}
    Given a pair $v,w\in V_f$, consider the map $a\mapsto B(av,w)$, which is a $k$-linear functional on $K_f$, and thereby corresponds to a unique element of $K_f$, which we denote $B_f(v,w)$. That is, choose the unique element satisfying $\Tr(aB_f(v,w)) = B(av,w)$.

    This assignment is $K_f$-bilinear. Given $b\in K_f$, for all $a\in K_f$, $\Tr(aB_f(bv,w)) = B(abv, w) = \Tr(abB_f(v,w))$. By nondegeneracy of the trace pairing, $B_f(bv,w) = bB_f(v,w)$, giving linearity in the first entry. Also, $\Tr(aB_f(v,bw)) = B(av,bw) = B(abv, w) = \Tr(abB_f(v,w))$ for all $a\in K_f$, so again by nondegeneracy, $B_f(v,bw) = b B_f(v,w)$, giving linearity in the second entry.

    The bilinear form $B_f$ is nondegenerate. Suppose that $B_f(v,w) = 0$ for all $w\in V_f$. Then taking $a = 1$, we have $\Tr(B_f(v,w)) = B(v,w) = 0$ for all $w$, thus by nondegeneracy $v = 0$.

    Finally, $B_f$ is alternating. For all $a\in K_f$, $B(av,av) = 0$, and $B(av,av) = \Tr(a^2B_f(v,v))$. Since the map $x\mapsto x^2$ is an isomorphism of $K_f$, indeed $\Tr(b B_f(v,v)) = 0$ for all $b\in K_f$, so by nondegeneracy we get $B_f(v,v) = 0$.
\end{proof}

In light of the above lemma, we can consider the symplectic Lie algebra $\sp(V_f, K_f)$ with respect to $B_f$.

\begin{lemma}
    The nilpotent component $N_{X,f}$ of $X_f$ lies in $\sp(V_f, K_f)$.
\end{lemma}

\begin{proof}
    From nondegeneracy of the trace pairing, the relation \[
        \begin{aligned}
            \Tr(aB_f(N_{X,f} v,w)) &= B(aN_{X,f} v,w) = B(N_{X,f} av, w) = B(av, N_{X,f}w)\\
            &= \Tr(aB_f(v, N_{X,f}w))
        \end{aligned}
    \] for all $a\in K_f$ implies that $B_f(N_{X,f}v, w) = B_f(v, N_{X,f} w)$, as desired.\end{proof}

\section{Local Nilpotent Orbit Data Parametrization}

\subsection{Overview}

The purpose of this section is to decompose the group centralizer and Lie algebra centralizer of $X$ appearing in the summands of $(\ref{tag:1})$. We will ultimately show that the Lie algebra centralizer of $X$ decomposes as the direct sum of the $\sp(V_f(X))$-centralizers of the local nilpotent components $N_{X,f}$, and analogously for the group centralizer. The main technical observation is that the $\Sp(V)$-orbits of $\sp(V)$, over which we sum in $(\ref{tag:1})$, are equivalent to the data of the $\sp(V_f(X))$-orbits of the nilpotent parts $N_{X,f}$ of any chosen representative $X$ in the corresponding Lie groups $\Sp(V_f(X))$. From this reparameterization, the desired decompositions emerge easily.

\subsection{Parametrization By Local Nilpotent Orbit Data}

We now introduce an invariant of elements of $\sp(V)$ which will allow us to parametrize the $\Sp(V)$-orbits in $\sp(V)$. Given $X\in \sp(V)$, by Lemma~\ref{lemma:orthog} we obtain an orthogonal decomposition $V = \bigperp_f V_f(X)$ over the monic irreducible factors $f$ of the minimal polynomial of $X$. Each $V_f(X)$ admits a symplectic group $\Sp(V_f, K_f)$ and Lie algebra $\sp(V_f, K_f)$ with the symplectic form $B_f$. Moreover, the nilpotent component $N_{X,f}$ of $X_f$ is in $\sp(V_f)$. Let $\cO_f(X)$ denote the $\Sp(V_f)$-orbit of $N_{X,f}$. Suppose that $\dim_{K_f} V_f = 2m_f$, and $\deg(f) = d_f$. Then clearly $n = \sum_f d_f m_f$. We can package this data arising from $X$ as follows. 

\begin{defn}
    We call $\{(f, \cO_f)\}_f$ a \emph{local datum}, where $f$ ranges over monic irreducible polynomials $f\in k[t]$, if the following hold.
    
    \begin{itemize}
        \item[(1)] Each $\cO_f$ is a nilpotent $\Sp(2m_f, q^{d_f})$-orbit in $\sp(2m_f, q^{d_f})$.
        \item[(2)] For all $f$, the degree of $f$ is $d_f$, and $\sum_f m_f d_f = n$.
        \item[(3)] The datum has finite support: $m_f=0$ for all but finitely many $f$, and $\cO_f=\{0\}$ when $m_f=0$.
    \end{itemize}
    By the above discussion, it is clear that $X\in \sp(V)$ defines a local datum, which we denote $\sF_X=\{(f, \cO_f(X))\}_f$, the \emph{local datum associated to $X$}. Write $X\sim Y$ if $X$ and $Y$ lie in the same $\Sp(V)$-orbit.
\end{defn}

\begin{prop}\label{prop:3-1}
    The $\Sp(V)$-orbits of $\sp(V)$ are parametrized exactly by local data. That is, for $X$ and $Y$ in $\sp(V)$, $X\sim Y$ if and only if $\sF_X = \sF_Y$, and additionally each local datum $\{(f, \cO_f)\}_f$ can be realized as $\sF_Z$ for some $Z\in \sp(V)$.
\end{prop}

\begin{proof}
    First we show that $\sF_{(-)}$ is constant on $\Sp(V)$-orbits. To this end, it suffices to take $X,Y\in \sp(V)$ with $X\sim Y$, and show that for each monic irreducible $f\in k[t]$, there is a $K_f$-linear symplectic isomorphism $V_f(X) \cong V_f(Y)$ that sends $N_{X,f} \mapsto N_{Y,f}$.

    Suppose we can write $Y = gXg^{-1}$ for some $g\in \Sp(V)$. The restriction of $g$ to $V_f(X)$ defines a map $V_f(X) \to V_f(Y)$. Indeed, for $v\in V_f(X)$, $f(Y)^r(gv) = g f(X)^rv = 0$ for large $r$, so $gv\in V_f(Y)$. As $g$ is invertible, it therefore identifies $V_f(X)$ with $V_f(Y)$. 
    
    Now we show that this map is $K_f$-linear. Take an element $a = p(\alpha_f)$ in $K_f$. The action of $a$ on $V_f(X)$ is given by $p(S_{X, f})$, and the action of $a$ on $V_f(Y)$ is given by $p(S_{Y, f})$.
    Since $gX = Yg$, restriction to $V_f(X)$ gives $gX_f = Y_f g$. By uniqueness of the Jordan-Chevalley decomposition, one can check that $gS_{X,f} = S_{Y,f}g$ and $gN_{X,f} = N_{Y,f}g$. To check $K_f$-linearity of $g$, we obtain the following\[
        g(a\cdot v) = g(p(S_{X,f})v) = p(S_{Y,f})g(v) = a\cdot g(v).
    \]

    Finally, we check that $g$ preserves the symplectic forms. Let $B_{X,f}$ and $B_{Y,f}$ be the symplectic forms on $V_f(X)$ and $V_f(Y)$, respectively. For $v,w\in V_f(X)$ and for all $a\in K_f$, \[
        \begin{aligned}
            \Tr(aB_{Y,f}(gv, gw)) &= B(agv, gw) = B(gav, gw) = B(av, w)\\ &= \Tr(aB_{X,f}(v,w)).
        \end{aligned}
    \]
    By nondegeneracy of the trace pairing, $B_{Y,f}(gv, gw) = B_{X,f}(v,w)$, as desired. We have shown that $g$ induces a map $\sp(V_f(X), K_f) \to \sp(V_f(Y), K_f)$ which is a symplectic isomorphism sending $N_{X,f}$ to $N_{Y,f}$, which is enough to show that $\sF_X = \sF_Y$.

    \smallskip
    For the converse, suppose that $\sF_X = \sF_Y$. This means that for each $f$, there is a $K_f$-linear symplectic isomorphism $g_f:V_f(X) \to V_f(Y)$ satisfying $g_f N_{X,f} = N_{Y,f}g_f$. Moreover, $K_f$-linearity implies that $g_f S_{X,f} = S_{Y,f} g_f$. These two relations combine to give $g_f X_f = Y_f g_f$. By assumption, $B_{Y,f}(g_f v, g_fw) = B_{X,f}(v,w)$, so we get \[
        B(g_fv, g_fw) = \Tr(B_{Y,f}(g_fv, g_fw)) = \Tr(B_{X,f}(v,w)) = B(v,w).
    \]
    Define $g = \bigoplus_f g_f$ with respect to $V = \bigperp_f V_f(X) \cong \bigperp_f V_f(Y)$. By the above observation, $g\in \Sp(V)$. Since $g_f X_f = Y_f g_f$ on each summand, $gX = Yg$ and $X\sim Y$.

    \smallskip
    It remains to prove that given a local datum $D = \{(f, \cO_f)\}_f$, there is some $Z\in\sp(V)$ satisfying $\sF_Z = D$. 
    
    For each monic irreducible polynomial $f$ in the support of $D$, the datum $\cO_f$ is a nilpotent $\Sp(2m_f, q^{d_f})$-orbit in $\sp(2m_f, q^{d_f})$, where $d_f = \deg(f)$. Concretely, set $K_f = \F_{q^{d_f}}$, let $W_f$ be a $2m_f$-dimensional symplectic $K_f$-vector space, equipped with symplectic form $\widetilde B_f$, view $\cO_f$ as an orbit in $\sp(W_f,K_f)$, and choose a representative $e_f\in \cO_f$. The space $W_f$ can naturally be viewed as a $k$-symplectic vector space by introducing the form $B_f^{(0)}$ on $W_f$, defined by \[
        B_f^{(0)}(v,w) = \Tr(\widetilde B_f(v,w)).
    \]
    
    It is clear that $B_f^{(0)}$ is alternating as $\widetilde B_f$ is alternating. Take $v\ne 0$. Since $\widetilde B_f$ is nondegenerate, choose $w_0$ such that $\widetilde B_f(v, w_0) \ne 0$. Since the trace pairing is nondegenerate, choose $c\in K_f$ such that $\Tr(c\widetilde B_f(v, w_0))\ne 0$. Then \[
        B_f^{(0)}(v, cw_0) = \Tr(\widetilde B_f(v, cw_0)) = \Tr(c\widetilde B_f(v, w_0))\ne 0.
    \]
    This shows that $B_f^{(0)}$ is nondegenerate.

    Now define $Z_f\in \End_{k}(W_f)$ by $Z_f = \alpha_f I + e_f$, where $\alpha_f$ is a root of $f$ that generates $K_f$ over $k$. Moreover, $Z_f \in \sp(W_f, k)$ with respect to $B_f^{(0)}$. This is because $\alpha_f I$ is self-adjoint for $B_f^{(0)}$, and $e_f\in \sp(W_f, K_f)$ means that $\widetilde B_f(e_f v, w) = \widetilde B_f(v, e_f w)$, which implies that $B_f^{(0)}(e_f v, w) = B_f^{(0)}(v, e_f w)$ by taking trace. Now define $W = \bigperp_f W_f$ to be the orthogonal direct sum of the $W_f$, and equip $W$ with the form $B_W = \bigperp_f B_f^{(0)}$, and set $Z = \bigoplus_f Z_f$. Note that $\dim_{k}W = \sum_f \dim_{k}W_f = \sum_f 2d_f m_f = 2n$. Since every $2n$-dimensional symplectic vector space over $k$ is isometric to $V$, choose an isometry $W\cong V$, which identifies $Z$ as an element of $\sp(V)$. It remains to show that each $W_f = V_f(Z)$, and that the form $B_{Z,f}$ on $V_f(Z)$ obtained from Lemma~\ref{lemma:2-2} agrees with $\widetilde B_f$.

    On $W_f$, $Z|_{W_f} = Z_f = \alpha_f I + e_f$. Since $e_f$ is nilpotent, $f(Z_f) = f(\alpha_f I + e_f)$ is nilpotent, as the constant term of $f(\alpha_f + u)\in K_f[u]$ is $f(\alpha_f) = 0$. This shows that $W_f \subseteq V_f(Z)$. If $h\ne f$, then on $W_h$ we have $f(Z_h)=f(\alpha_h I+e_h)$. Its semisimple part is the nonzero scalar $f(\alpha_h)$, since the distinct irreducible polynomials $f$ and $h$ have no common root. Thus $f(Z_h)$ is invertible. Therefore $V_f(Z)=W_f$. Finally, consider the form $B_{Z,f}$, satisfying $\Tr(aB_{Z,f}(v,w)) = B_W(av,w)$. On $W_f$, $B_W = B_f^{(0)}$. Take $a\in K_f$ and $v,w\in W_f$. Then \[
        B_W(av, w) = B_f^{(0)}(av,w) = \Tr(\widetilde B_f(av, w)) = \Tr(a\widetilde B_f(v,w)). 
    \]
    But $B_{Z,f}(v,w)$ is defined as the unique quantity satisfying $\Tr(aB_{Z,f}(v,w)) = B_W(av,w)$ for all $a\in K_f$, so the two forms agree.
\end{proof}

\begin{prop}\label{prop:decomp}
    Let $X\in \sp(V)$, and let $N_{X,f}\in \sp_{2m_f}(K_f)$ be its local nilpotent parts. Then\[
        C_{\sp(V)}(X) \cong \bigoplus_f C_{\sp_{2m_f}(K_f)}(N_{X,f})
    \]
    as $k$-vector spaces, and \[
        C_{\Sp(V)}(X) \cong \prod_f C_{\Sp_{2m_f}(K_f)}(N_{X,f})
    \]
    as groups.
\end{prop}

\begin{proof}
    If $Y\in C_{\sp(V)}(X)$, then $Y$ preserves each primary summand $V_f(X)$ and commutes with both $S_{X,f}$ and $N_{X,f}$. Since the action of $K_f$ on $V_f(X)$ is generated by $S_{X,f}$, the restriction $Y_f=Y|_{V_f(X)}$ is $K_f$-linear. For all $a\in K_f$ and $v,w\in V_f(X)$,
    \[
        \Tr(aB_f(Y_fv,w))
        =B(aY_fv,w)
        =B(Y_fav,w)
        =B(av,Y_fw)
        =\Tr(aB_f(v,Y_fw)).
    \]
    Hence $Y_f\in C_{\sp(V_f,K_f)}(N_{X,f})$. This gives an injective map
    \[
        C_{\sp(V)}(X)\longrightarrow
        \bigoplus_f C_{\sp(V_f,K_f)}(N_{X,f}).
    \]
    Conversely, if $Y_f\in C_{\sp(V_f,K_f)}(N_{X,f})$ for every $f$, then $Y=\bigoplus_fY_f$ commutes with each $S_{X,f}+N_{X,f}=X_f$. Applying the trace identity and using the orthogonality of the primary decomposition shows that $Y$ is self-adjoint for $B$. Thus $Y\in C_{\sp(V)}(X)$, which proves the first isomorphism.

    The group case is similar. If $g\in C_{\Sp(V)}(X)$, then $g$ preserves each $V_f(X)$, is $K_f$-linear there, and commutes with $N_{X,f}$. Moreover, for all $a\in K_f$ and $v,w\in V_f(X)$,
    \[
        \Tr(aB_f(gv,gw))
        =B(agv,gw)
        =B(gav,gw)
        =B(av,w)
        =\Tr(aB_f(v,w)).
    \]
    Hence $g|_{V_f(X)}\in C_{\Sp(V_f,K_f)}(N_{X,f})$. Conversely, an element of $\prod_f C_{\Sp(V_f,K_f)}(N_{X,f})$ preserves $B$ by the same trace identity and orthogonality, and it commutes with $X$ on every primary summand. This proves the second isomorphism.
\end{proof}

\section{A Generating Function for Commuting Pairs}

We are now able to present a generating function for the normalized quantities \[
    \overline N_n(q) = \frac{N_n(q)}{|\Sp_{2n}(\F_q)|}.
\]
For $q$ a power of $2$, set $k = \F_q$, and we define $A_m(q)$ as follows. Set $A_0(q) = 1$, and for $m\ge 1$, set \[
    A_m(q) = \sum_{\cO}\frac{|C_{\sp_{2m}(\F_q)}(N)|}{|C_{\Sp_{2m}(\F_q)}(N)|},
\]
taken over nilpotent $\Sp_{2m}(k)$-orbits $\cO$ in $\sp_{2m}(k)$, with $N$ any representative of $\cO$. Let $I(d)$ denote the number of irreducible polynomials in $k[t]$ of degree $d$. Recall that $I(d)$ has the elementary formula\[
    I(d) = \frac1d \sum_{k\mid d} \mu(k) q^{\frac dk},
\]
where $\mu$ is the Möbius function.

\begin{thm}\label{thm:euler}
    For $q$ a power of $2$, \[
    \begin{aligned}
        \sum_{n \ge 0}\overline{N}_n(q) u^n &= \prod_f \left(\sum_{m\ge 0} A_m(q^{\deg(f)}) u^{m\deg(f)}\right) \\
        &= \prod_{d\ge 1}\left(\sum_{m\ge 0} A_m(q^{d}) u^{md}\right)^{I(d)}
    \end{aligned}\tag{2}\label{tag:2}
    \]    
    taken over monic irreducible $f\in k[t]$.
\end{thm}

\begin{proof}
    Recall from $(\ref{tag:1})$ that we may write \[
        \overline N_n(q) = \sum_{[X]}\frac{|C_{\sp_{2n}(\F_q)}(X)|}{|C_{\Sp_{2n}(\F_q)}(X)|}
    \]
    taken over $\Sp_{2n}(k)$-orbits in $\sp_{2n}(k)$. By the parametrization of such orbits by local data given in Proposition~\ref{prop:3-1}, an orbit $[X]$ is equivalent to a local datum $\{(f, \cO_f)\}_f$ such that $\cO_f$ is a nilpotent $\Sp_{2m_f}(q^{d_f})$-orbit in $\sp_{2m_f}(q^{d_f})$, $d_f = \deg(f)$, and $\sum d_f m_f = n$. Fixing an orbit and corresponding local datum, choosing representatives $N_f\in \cO_f$, by Proposition~\ref{prop:decomp}, the centralizers decompose as \[
        |C_{\sp_{2n}(q)}(X)| = \prod_f |C_{\sp_{2m_f}(q^{d_f})}(N_f)|
    \]
    and \[
        |C_{\Sp_{2n}(q)}(X)| = \prod_f |C_{\Sp_{2m_f}(q^{d_f})}(N_f)|.
    \]
    Therefore the contribution to $\overline N_n (q)$ from this orbit is \[
        \prod_f \frac{|C_{\sp_{2m_f}(q^{d_f})}(N_f)|}{|C_{\Sp_{2m_f}(q^{d_f})}(N_f)|}.
    \]
    The coefficient of $u^n$ in the middle expression of $(\ref{tag:2})$ is the sum over all products of terms $A_{m_i}(q^{d_i})$ such that $\sum d_i m_i = n$. But by definition of $A_m(q)$, this is the same as summing over all admissible parameters $\{(d_f, m_f)\}$ of local data, and for each taking the contribution from all local data with the given parameters. This shows the first equality. The second equality is immediate from the fact that each factor in the middle expression depends only on the degree of $f$, and thereby grouping the polynomials by degree.
\end{proof}
    
\section{Evaluation of the Local Contribution}

The generating function of Theorem~\ref{thm:euler} is not yet of much use, since the quantities $A_m(q)$ have no obvious way of being computed. In this section, we compute these values, dependent on the structure of the so-called \emph{form module} built from $V$. The notion of a form module, defined below, is due to Hesselink in \cite{Hesselink1979}. For more information about the form module, and the decomposition we use below, see \cite[\S 5]{LiebeckSeitz}. 

\medskip

Fix $q$ a power of $2$, $K = \F_q$, $V$ a $2m$-dimensional symplectic $K$-vector space with symplectic form $B$. Let $G_K = \Sp(V)$ and $\frg_K = \sp(V)$. In this section we also consider $V_{\overline K} = V\otimes_K \overline K$, noting that $B$ extends to a symplectic form on $V_{\overline K}$, and set $\sG = \Sp(V_{\overline{K}})$ and $\frg = \sp(V_{\overline{K}})$. Let $F$ be the Frobenius map on $\sG$ giving the $K$-structure, that is, $F$ is the Steinberg endomorphism such that $\sG^F = G_K$ and $\frg^F = \frg_K$. To avoid confusion, $\dimvar$ will denote the dimension of a variety, whereas $\dim_K$ or $\dim_{\overline K}$ denote vector space dimension. 

First we recall a standard fact about Lie algebra centralizers, and give a short proof for completeness.

\begin{lemma}\label{lemma:5-1}
    Let $e\in \frg_K$. Then the Lie algebra centralizer satisfies \[
        C_\frg(e) = C_{\frg_K}(e) \otimes_K \overline K.
    \]
    Consequently, $C_\frg(e)^F = C_{\frg_K}(e)$, and thus \[
        |C_\frg(e)^F| = q^{\dim_{\overline K}C_\frg(e)}.
    \]
\end{lemma}

\begin{proof}
    Let $\ad_{\frg_K}(e):\frg_K\to \frg_K$ be the adjoint map $X\mapsto [X,e]$, so $C_{\frg_K}(e) = \ker(\ad_{\frg_K}(e))$, and similarly $C_\frg(e) = \ker( \ad_\frg(e))$. Noting that $\frg = \frg_K \otimes_K \overline K$, after extending scalars, it is also true that $\ad_{\frg_K}(e) \otimes_K \overline K = \ad_\frg(e)$. Since $\overline K$ is flat over $K$, it follows that $C_\frg(e) = C_{\frg_K}(e)\otimes_K \overline K$. The second claim follows immediately.
\end{proof}

In this notation, we may write \[
    A_m(q) = \sum_\cO \frac{|C_\frg(e_\cO)^F|}{|C_\sG(e_\cO)^F|}
\]
taken over nilpotent $\sG^F$-orbits $\cO$ in $\frg^F$ with representatives $e_\cO\in \cO$. Due to Lemma~\ref{lemma:5-1}, this quantity is \[
    A_m(q) = \sum_\cO \frac{q^{\dim_{\overline K}(C_\frg(e_\cO))}}{|C_\sG(e_\cO)^F|}.
\]

The results of Hesselink \cite{Hesselink1979} and Liebeck--Seitz \cite{LiebeckSeitz} allow us to compute this quantity. We now introduce some terminology used by both authors.

Let $e\in \frg$ be a nilpotent element. View $V_{\overline K}$ as a $\overline K[t]$-module by allowing $t$ to act as $e$. Then the $\overline K[t]$-submodules are $e$-stable subspaces. The symplectic relation says that $B(tv,w) = B(v,tw)$, thus the structure as a $\overline K[t]$-module is compatible with the symplectic form. Let $V_{\overline K} \downarrow e$ denote the data of $(V_{\overline K}, B, e)$ viewed together as a compatible \emph{form module}. An orthogonal decomposition of $V_{\overline K} \downarrow e$ means a decomposition $V_{\overline K} = \bigoplus_s U_s$ such that each $U_s$ is $e$-stable, the restriction $B|_{U_s}$ of $B$ to $U_s$ is nondegenerate, and $B(U_s, U_r) = 0$ for $r\ne s$. 

The following two functions are used in \cite{Hesselink1979}. For $n\ge 0$, define $\alpha_n: V_{\overline K}\to \overline K$ by $\alpha_n(v) = B(e^{n+1}v, e^nv)$. For $r\ge 1$, define \[
    \chi_e(r) = \min\{n\ge 0~\mid~ e^r v = 0 \implies \alpha_n(v) = 0~\forall v\in V_{\overline K}\}.
\]
For a nondegenerate $e$-stable subspace $U \subset V_{\overline K}$, we also obtain a function $\chi_U$ by the same definition, considering $U$, $B|_U$, $e|_U$, and $U\downarrow e|_U$. These $\chi$ functions are called \emph{index functions}, and the following lemma allows us to write the index function as the maximum of those arising from its orthogonal summands.

\begin{lemma}\label{lemma:5-2}
    Let $\bigperp_s U_s$ be an orthogonal decomposition of $V_{\overline K}\downarrow e$. Then $\chi_e(r) = \max_s \{\chi_{U_s}(r)\}$ for every $r\ge 1$.
\end{lemma}

\begin{proof}
    Fix $r \ge 1$. We show that $n$ satisfies $n\ge \chi_e(r)$ if and only if $n\ge \chi_{U_s}(r)$ for all $U_s$, from which the result follows.

    Suppose that $n\ge \chi_{U_s}(r)$ for all $U_s$, that is $e^r u = 0 \implies \alpha_n(u) = 0$ for all $u\in U_s$. Let $v \in V_{\overline K}$ satisfy $e^r v = 0$. Write $v = \sum_s v_s$ with $v_s \in U_s$. Since $U_s$ is $e$-stable, $e^r v_s = 0$ for all $s$. Since $e$ is self-adjoint with respect to $B$, $B(e^a U_s, e^b U_r) = 0$ for all $s\ne r$, and therefore $B(e^{n+1}v, e^n v) = \sum_s B(e^{n+1}v_s, e^n v_s) = 0$. Thus $n\ge \chi_e(r)$.

    Suppose that $n\ge \chi_e(r)$. Fix $s$, and let $u\in U_s$ satisfy $e^r u = 0$. Since $u\in V_{\overline K}$, $B(e^{n+1}u, e^n u) = 0$, so $n \ge \chi_{U_s}(r)$.
\end{proof}

For $a\ge 1$ and $0 \le l\le a$, define the function $[a;l]:\mathbb N\to \Z$ by $r\mapsto \max\{0, \min\{r - a + l, l\}\}$. The following lemma is a consequence of \cite[Table 4.1]{LiebeckSeitz}. The modules $V(2k)$, $W(a)$, and $W_l(a)$ are defined in \cite[\S 5.1]{LiebeckSeitz}. We use the parameter $l$ in the convention of that source, without reindexing.

\begin{lemma}\label{lemma:5-3}
    For a nilpotent element $e\in \frg$, if $U$ is an indecomposable summand appearing in an orthogonal decomposition of $V_{\overline K}\downarrow e$, then $U$ is isomorphic to one of $V(2k)$, $W(a)$, $W_l(a)$, with the following characteristics.\begin{itemize}
        \item[(1)] If $U\cong V(2k)$, then $\dim_{\overline K}U = 2k$, $e|_U$ is a single Jordan block of size $2k$, and $\chi_U = [2k;k]$.
        \item[(2)] If $U\cong W(a)$, then $\dim_{\overline K}U = 2a$, $e|_U$ has two Jordan blocks of size $a$, and $\chi_U = [a;0]$.
        \item[(3)] If $U\cong W_l(a)$, where $0<l<a/2$, then $\dim_{\overline K}U = 2a$, $e|_U$ has two Jordan blocks of size $a$, and $\chi_U = [a;l]$.
    \end{itemize}
\end{lemma}

\begin{defn}
    Let $\mathcal D_m$ be the set of formal sums \[
    \sigma = \sum_i W(m_i)^{a_i} + \sum_j W_{l_j}(n_j) + \sum_r V(2k_r)^{c_r}
\]
such that the following conditions hold. \begin{itemize}
    \item[(1)] $\sum_i 2a_i m_i + \sum_j 2n_j + \sum_r 2c_r k_r = 2m$.
    \item[(2)] $a_i > 0$, $m_i \ne m_j$ for $i\ne j$, $0 < l_j < n_j / 2$.
    \item[(3)] The sequences $\{n_j\}$, $\{l_j\}$, $\{n_j - l_j\}$ are strictly decreasing.
    \item[(4)] Each $c_r \in \{1,2\}$.
    \item[(5)] For all $j,r$, $k_r > n_j - l_j$ or $k_r < l_j$.
\end{itemize}
\end{defn}

The elements $\sigma$ of $\cD_m$ are called \emph{distinguished normal forms}.

\begin{lemma}\label{lemma:5-5}
    The set $\cD_m$ parametrizes the nilpotent $\sG$-orbits in $\frg$. That is, the following hold. \begin{itemize}
        \item[(1)] For every nilpotent $e\in \frg$, the form module $V_{\overline K}\downarrow e$ has an associated normal form datum $\sigma(e)\in \cD_m$.
        \item[(2)] For every $\sigma \in \cD_m$, there exists a nilpotent $e\in \frg$ such that $\sigma(e) = \sigma$.
        \item[(3)] If $e, e'\in \frg$ are nilpotent elements, then $e$ and $e'$ are $\sG$-conjugates if and only if $\sigma(e) = \sigma(e')$.
        \item[(4)] If $\sigma\in \cD_m$ and $e\in \frg$ satisfies $\sigma(e) = \sigma$, then the index function $\chi_e$ of $V_{\overline K}\downarrow e$ is \[
            \chi_e(s) = \max\{[m_i; 0](s), [n_j; l_j](s), [2k_r; k_r](s)\},
        \]
        taken over summands in $\sigma$.
    \end{itemize}
\end{lemma}

\begin{proof}
    Every nilpotent $e\in \frg$ gives rise to an orthogonal decomposition $\sigma(e)\in \cD_m$ of $V_{\overline K} \downarrow e$. The decomposition is written in \cite[Eq.~(5.5)]{LiebeckSeitz}, and the conditions satisfied by the parameters of decompositions in $\cD_m$ are given in \cite[Prop.~5.3(ii)]{LiebeckSeitz}. The converse similarly follows from \cite[\S 5]{LiebeckSeitz}.

    For the third point, let $e, e'\in \frg$. First suppose that $e\sim e'$, so $e' = geg^{-1}$ for some $g\in \sG$. Then $g$ defines a map $V_{\overline K} \to V_{\overline K}$ which is a symplectic isometry, and moreover $g(ev) = e'g(v)$. This means that $g$ is an isomorphism of form modules $V_{\overline{K}}\downarrow e\to V_{\overline{K}}\downarrow e'$, so $\sigma(e) = \sigma(e')$. Here we are using uniqueness of the decomposition of the form module, \cite[Lemma 5.4]{LiebeckSeitz}. The converse is nearly identical. Suppose that $\sigma(e) = \sigma(e')$. Then the isomorphism $V_{\overline{K}}\downarrow e \to V_{\overline{K}}\downarrow e'$ is given by an element $g\in \sG$, and by assumption $ge = e'g$, so $e\sim e'$.

    The fourth point follows from Lemmas~\ref{lemma:5-2} and~\ref{lemma:5-3}. \end{proof}

The previous lemma shows that over $\overline K$, the nilpotent $\sG$-orbits in $\frg$ are in correspondence with formal decomposition of the associated form modules $V_{\overline K}\downarrow e$, where $e$ and $e'$ are $\sG$-conjugate exactly if they give rise to isomorphic form modules, which necessarily have the same formal decomposition. The next lemma shows that indeed every nilpotent orbit contains a representative defined over $K$.

\begin{lemma}\label{lemma:5-6}
    Let $e\in \frg$ be a nilpotent element. Then there exists a nilpotent $e' \in \frg_K$ such that $e$ and $e'$ are $\sG$-conjugates.
\end{lemma}

\begin{proof}
    By the correspondence given in Lemma~\ref{lemma:5-5}, it suffices to prove that for every distinguished normal form $\sigma\in \cD_m$, there exists a rational representative $e_\sigma \in \frg$ such that $V_{\overline K} \downarrow e_\sigma$ has orthogonal decomposition given by the expression $\sigma$. Recall that $V_{\overline K}\downarrow e$, for $e\in \frg_K$, is formed as the data $(V\otimes_K \overline K, e\otimes 1, B_{\overline K})$.

    The constructions of the modules $V(2k)$, $W(a)$, and $W_l(a)$ are given explicitly in \cite[\S 5.1]{LiebeckSeitz}, and can be taken over $K$. So construct the form module over $K$ corresponding to the expression $\sigma$, by taking the orthogonal direct sum of the specified indecomposable form modules, and let $V_\sigma$ denote the $2m$-dimensional symplectic $K$-vector space. Let $e_\sigma^0$ denote the nilpotent operator given by taking the orthogonal direct sum of the nilpotent operator on each indecomposable summand. Then we obtain $e_\sigma \in \sp(V)$ by the isometry $V_\sigma \to V$. It is clear that, after extension of scalars to $\overline K$, $V_{\overline K} \downarrow e_\sigma$ has distinguished normal form $\sigma$. 
\end{proof}

In light of the previous lemma, fix once and for all the representative $e_\sigma\in\frg_K$ constructed in the proof of Lemma~\ref{lemma:5-6} for each $\sigma\in\cD_m$. Let $t_1\ge\dots\ge t_s$ denote the Jordan block sizes of $e_\sigma$ on $V_{\overline K}$, and let $\chi_\sigma$ denote the index function of $V_{\overline K}\downarrow e_\sigma$.

\begin{lemma}\label{lemma:5-7}
    Let $\sigma\in\cD_m$, and set $C_\sigma=C_\sG(e_\sigma)$ and $U_\sigma=R_u(C_\sigma)$. Then the following hold.
    \begin{itemize}
        \item[(1)] The algebraic group $C_\sigma$ is connected.
        \item[(2)] The dimension of $C_\sigma$ is
        \[
            \dimvar C_\sigma=\sum_{j=1}^s\left(jt_j-\chi_\sigma(t_j)\right).
        \]
        \item[(3)] There exists an $F$-stable reductive subgroup $R_\sigma\le C_\sigma$, defined over $K$, such that
        \[
            C_\sigma=U_\sigma\rtimes R_\sigma
        \]
        and
        \[
            R_\sigma\cong\prod_i\Sp_{2a_i}
        \]
        as algebraic groups over $K$, where the product is taken over the summands $W(m_i)^{a_i}$ of $\sigma$. Consequently,
        \[
            R_\sigma^F\cong\prod_i\Sp_{2a_i}(K).
        \]
    \end{itemize}
\end{lemma}

\begin{proof}
    The first assertion follows from \cite[Thm.~5.12(i)]{LiebeckSeitz}, and the second follows from \cite[Lemma~5.4]{LiebeckSeitz}.

    Use the $K$-linear isometry chosen in the proof of Lemma~\ref{lemma:5-6} to identify $V_\sigma$ with $V$. Under this identification, the restrictions of $B$ and $e_\sigma$ to each indecomposable summand are given by the bases and formulas in \cite[\S 5.1]{LiebeckSeitz}.

    Using these bases, define a one-dimensional torus $T=\{T(c)\mid c\in\overline K^\times\}\leq\sG$ by letting $T(c)$ act on a basis vector with index $i$ as multiplication by $c^i$. The formulas in \cite[\S 5.1]{LiebeckSeitz} show that $T$ preserves the symplectic form and that $\operatorname{Ad}(T(c))(e_\sigma)=c^2e_\sigma$. For $i\in\mathbb Z$, let
    \[
        V_i=\{v\in V_{\overline K}\mid T(c)v=c^iv\text{ for all }c\in\overline K^\times\}
    \]
    denote the weight-$i$ space of $T$. The matrices defining $T$ in these bases have entries in $K[c,c^{-1}]$, so $T$ is defined over $K$. Moreover, each $V_i$ has a basis consisting of vectors in $V$.
    
    Fix a summand $Z_i=W(m_i)^{a_i}$ of $V_\sigma$. When $m_i$ is even, the construction in \cite[Lemma~5.7]{LiebeckSeitz} identifies $Z_i$ with a tensor product $M_i\otimes A_i$, where $\dim M_i=m_i$, $\dim A_i=2a_i$, and $e_\sigma$ and $T$ act only on $M_i$. Therefore the subgroup acting trivially on $M_i$ and as $\Sp(A_i)$ on $A_i$ is a subgroup $R_i\cong\Sp_{2a_i}$ centralizing $e_\sigma$ and $T$. In the specified basis, this action is defined over $K$. If $m_i$ is odd, the same construction begins with the natural action of $\Sp_{2a_i}$ on the zero-weight space of $Z_i$ and extends this action to the other weight spaces using the maps induced by $e_\sigma$. Since these weight spaces and maps are defined over $K$, the resulting subgroup $R_i\cong\Sp_{2a_i}$ is also defined over $K$. In either case, $R_i$ centralizes $e_\sigma$ and acts trivially on every other indecomposable summand.
    
    Therefore the subgroups $R_i$ commute, and $R_\sigma=\prod_iR_i$ is an $F$-stable reductive subgroup of $C_\sigma$, defined over $K$, such that $R_\sigma\cong\prod_i\Sp_{2a_i}$ over $K$. By \cite[Thm.~5.6(i) and Prop.~5.11]{LiebeckSeitz}, this subgroup splits over the unipotent radical. Hence $C_\sigma=U_\sigma\rtimes R_\sigma$. Since each factor $R_i$ is the split symplectic group defined over $K$, taking fixed points gives
    \[
        R_\sigma^F\cong\prod_i\Sp_{2a_i}(K).
    \]
\end{proof}

\begin{lemma}[Thm. 4.4 \cite{Hesselink1979}]\label{lemma:5-8}
    Let $e\in \frg$ be nilpotent, and $t_1\ge\dots\ge t_s$ denote the Jordan block sizes of $e$ on $V_{\overline{K}}$. Then \[
        \dim_{\overline K} C_\frg (e) = \sum_{j = 1}^s jt_j.
    \]
\end{lemma}

We will need the following consequence of the Lang-Steinberg Theorem, given as \cite[Thm. 21.11]{MalleTesterman2011} by Malle-Testerman. Recall that for an algebraic group $\sH$, the connected component containing the identity is denoted $\sH^\circ$.

\begin{lemma}\label{lemma:5-9}
    Let $\sH$ be a connected linear algebraic group with Steinberg endomorphism $F:\sH\to \sH$. Suppose that $\sH$ acts transitively on $\sY$, a nonempty variety with a compatible $F$-action $F:\sY \to \sY$. For $y\in \sY$, let $\sH_y$ denote the stabilizer of $y$ in $\sH$, and assume that $\sH_y\le \sH$ is closed. Then the $\sH^F$-orbits on $\sY^F$ are parametrized by the $F$-classes in $\sH_y / \sH_y^\circ$.
\end{lemma}

\begin{lemma}\label{lemma:5-10}
    For each $\sigma\in \cD_m$, the fixed point set $(\sG \cdot e_\sigma)^F$ is a single $\sG^F$-orbit. Consequently, the nilpotent $\sG^F$-orbits in $\frg^F$ are uniquely represented by the elements $e_\sigma$ for $\sigma\in \cD_m$.
\end{lemma}

\begin{proof}
    For $\sigma \in \cD_m$, let $\sY_\sigma$ denote the orbit $\sG\cdot e_\sigma$. Since $e_\sigma\in \frg^F$, $\sY_\sigma$ is stable under $F$. Thus $F$ restricts to a map on $\sY_\sigma$. By construction, the action of $\sG$ on $\sY_\sigma$ is transitive. The stabilizer $C_\sG(e_\sigma)$ is closed in $\sG$. Applying Lemma~\ref{lemma:5-9}, the $\sG^F$-orbits on $\sY_\sigma^F$ are parametrized by $F$-classes in $C_\sG(e_\sigma) / C_\sG(e_\sigma)^\circ$. By Lemma~\ref{lemma:5-7}(1), $C_\sG(e_\sigma)$ is connected, so the quotient is trivial. Thus $\sY_\sigma^F = (\sG \cdot e_\sigma)^F$ is a single $\sG^F$-orbit.

    Let $e\in \frg^F$ be nilpotent. Then viewing $e\in \frg$, obtain $\sigma(e) \in \cD_m$ by Lemma~\ref{lemma:5-5}. By Lemma~\ref{lemma:5-6}, $e_{\sigma(e)}$ has the same normal form datum as $e$. Therefore, by Lemma~\ref{lemma:5-5}, $e\in \sG\cdot e_{\sigma(e)}$. Since $e\in \frg^F$, $e\in (\sG\cdot e_{\sigma(e)})^F$, which is a single $\sG^F$-orbit, so $e$ and $e_{\sigma(e)}$ are $\sG^F$-conjugates. Conversely, suppose that $e_\sigma$ and $e_\tau$ are conjugates by $\sG^F$. Then by Lemma~\ref{lemma:5-5}, $\sigma = \tau$.
\end{proof}

\begin{lemma}[Lemma 1.7\cite{LawtherLiebeckSeitz2002FPR}]\label{lemma:5-11}
    Let $\sU$ be a connected unipotent algebraic group over $K$ with Frobenius map $F$. Then $|\sU^F| = q^{\dimvar \sU}$.
\end{lemma}

\begin{lemma}\label{lemma:5-12}
    Let $\sigma\in\cD_m$. Then
    \[
        |C_\sG(e_\sigma)^F|
        =
        q^{\dimvar C_\sG(e_\sigma)-\dimvar R_\sigma}|R_\sigma^F|,
    \]
    where
    \[
        R_\sigma^F\cong\prod_i\Sp_{2a_i}(K).
    \]
\end{lemma}

\begin{proof}
    Set $C=C_\sG(e_\sigma)$ and $U=R_u(C)$. By Lemma~\ref{lemma:5-7}, $C=U\rtimes R_\sigma$, where both $U$ and $R_\sigma$ are $F$-stable. Every $c\in C$ can therefore be written uniquely as $c=ur$, where $u\in U$ and $r\in R_\sigma$. If $c\in C^F$, then $ur=c=F(c)=F(u)F(r)$. By uniqueness, $F(u)=u$ and $F(r)=r$. Hence $C^F=U^F\rtimes R_\sigma^F$, and therefore $|C^F|=|U^F||R_\sigma^F|$. Since $U$ is connected and unipotent, Lemma~\ref{lemma:5-11} gives $|U^F|=q^{\dimvar U}$. Finally, the decomposition $C=U\rtimes R_\sigma$ gives $\dimvar U=\dimvar C-\dimvar R_\sigma$, and the result follows.
\end{proof}

The next lemma follows from the contents of Table $24.1$ of \cite{MalleTesterman2011}.

\begin{lemma}\label{lemma:5-13}
    The order of a symplectic group over $\F_q$ is \[
        |\Sp_{2a}(\F_q)| = q^{a^2} \prod_{h = 1}^a (q^{2h} - 1).
    \]
    Consequently,
    \[
        |R_\sigma^F| = \prod_i \left(q^{a_i^2} \prod_{h = 1}^{a_i} (q^{2h} - 1)\right).
    \]
\end{lemma}

\begin{thm}\label{thm: 5-1}
    The quantities $A_m(q)$ can be written \[
        A_m(q) = \sum_{\sigma \in \cD_m}\frac{q^{\sum_{j = 1}^s \chi_\sigma(t_j) + \sum_i a_i(2a_i + 1)}}{\prod_i\left(q^{a_i^2}\prod_{h = 1}^{a_i}(q^{2h} - 1)\right)}.
    \]
\end{thm}

\begin{proof}
    We may write \[
        A_m(q) = \sum_\cO\frac{|C_\frg(e_\cO)^F|}{|C_\sG(e_\cO)^F|} = \sum_{\sigma\in \cD_m} \frac{|C_\frg(e_\sigma)^F|}{|C_\sG(e_\sigma)^F|},
    \]
    where the first equality is by definition, and the second is by Lemma~\ref{lemma:5-10}. For $\sigma \in \cD_m$, Lemma~\ref{lemma:5-1} gives \[
        |C_\frg(e_\sigma)^F| = q^{\dim_{\overline K}C_\frg(e_\sigma)}.
    \]
    By Lemma~\ref{lemma:5-8}, \[
        \dim_{\overline K}C_\frg(e_\sigma) = \sum_{j =1}^s jt_j,
    \] and by Lemma~\ref{lemma:5-7}, 
    \[
        \dimvar C_\sG(e_\sigma) = \sum_{j = 1}^s (jt_j - \chi_\sigma(t_j)).
    \]
    By Lemma~\ref{lemma:5-12}, \[
        |C_\sG(e_\sigma)^F| = q^{\dimvar C_\sG(e_\sigma) - \dimvar R_\sigma} |R_\sigma^F|.
    \]
    Therefore \[
       \frac{|C_\frg(e_\sigma)^F|}{|C_\sG(e_\sigma)^F|} = \frac{q^{\dim_{\overline K}C_\frg(e_\sigma)}}{q^{\dimvar C_\sG(e_\sigma) - \dimvar R_\sigma} |R_\sigma^F|} = \frac{q^{\sum_{j = 1}^s \chi_\sigma(t_j) + \dimvar R_\sigma}}{|R_\sigma^F|}.
    \]
    The result follows by evaluating the denominator using Lemma~\ref{lemma:5-13}, and using $\dimvar \Sp_{2a} = a(2a + 1)$ in the numerator.
\end{proof}

\section{An Example Computation for $\sp_4$}

In this section, we compute the smallest nontrivial case of Theorem~\ref{thm:euler} using Theorem~\ref{thm: 5-1}. That is, we compute the number of commuting pairs in $\sp_4(q)\times \sp_4(q)$.

By Theorem~\ref{thm:euler}, the coefficient of $u^2$ in the generating function is
\[
    \frac{N_2(q)}{|\Sp_{4}(\F_q)|}
    =
    \binom{I(1)}{2}A_1(q)^2 + I(1)A_2(q) + I(2)A_1(q^2).
\]
Note that $I(1) = q$, $I(2) = \frac12(q^2 - q)$, and
\[
    |\Sp_4(\F_q)| = q^4(q^2 - 1)(q^4 - 1).
\]
So it remains to compute $A_1(q)$, $A_2(q)$, and $A_1(q^2)$. In the expression for $A_m(q)$ in Theorem~\ref{thm: 5-1}, we sum over admissible normal forms $\sigma\in \cD_m$. In the case $A_1(q)$, these are only $W(1)$ and $V(2)$, and in the case $A_2(q)$, these are $ W(1)^2, W(2), W(1)+V(2), V(2)^2, V(4)$. For each of these summands, the data that needs to be determined is the size of the Jordan blocks of $e_\sigma$ on $V_{\overline K}\downarrow e_\sigma$, the index function $\chi_\sigma$, and the reductive quotient $R_\sigma$. The Jordan block data and index functions are provided by Lemma~\ref{lemma:5-3}.

Let us compute $A_1(q)$. If $\sigma = W(1)$, the Jordan blocks are $(1,1)$, the index function is $\chi_\sigma = [1;0] = 0$, and $R_\sigma = \Sp_2$. Hence the contribution is
\[
    \frac{q^{\dim \Sp_2}}{|\Sp_2(\F_q)|} = \frac{q^3}{q(q^2 - 1)} = \frac{q^2}{q^2 - 1}.
\]
If $\sigma = V(2)$, the Jordan blocks are $(2)$, the index function is $\chi_\sigma = [2;1]$, so $\chi_\sigma(2) = 1$, and $R_\sigma = 1$. Therefore the contribution is $q$. Thus
\[
    A_1(q) = \frac{q^2}{q^2 - 1} + q.
\]

Now let us compute $A_2(q)$. If $\sigma = W(1)^2$, the Jordan blocks are $(1,1,1,1)$, the index function is $\chi_\sigma = [1;0] = 0$, and $R_\sigma = \Sp_4$. Hence the contribution is
\[
    \frac{q^{\dim \Sp_4}}{|\Sp_4(\F_q)|} = \frac{q^{10}}{q^4(q^2 - 1)(q^4 - 1)} = \frac{q^6}{(q^2 - 1)(q^4 - 1)}.
\]
If $\sigma = W(2)$, the Jordan blocks are $(2,2)$, the index function is $\chi_\sigma = [2;0] = 0$, and $R_\sigma = \Sp_2$. Hence the contribution is
\[
    \frac{q^{\dim \Sp_2}}{|\Sp_2(\F_q)|} = \frac{q^3}{q(q^2 - 1)} = \frac{q^2}{q^2 - 1}.
\]
If $\sigma = W(1)+V(2)$, the Jordan blocks are $(2,1,1)$, and the index function is
\[
    \chi_\sigma = \max\{[1;0],[2;1]\} = [2;1].
\]
Thus
\[
    \chi_\sigma(2)+\chi_\sigma(1)+\chi_\sigma(1)=1+0+0=1.
\]
The reductive quotient is again $R_\sigma = \Sp_2$. Hence the contribution is
\[
    \frac{q^{1+\dim \Sp_2}}{|\Sp_2(\F_q)|} = \frac{q^4}{q(q^2 - 1)} = \frac{q^3}{q^2 - 1}.
\]
If $\sigma = V(2)^2$, the Jordan blocks are $(2,2)$, and the index function is $\chi_\sigma = [2;1]$. Hence
\[
    \chi_\sigma(2)+\chi_\sigma(2)=1+1=2.
\]
There is no reductive quotient, so $R_\sigma = 1$, and the contribution is $q^2$.
If $\sigma = V(4)$, the Jordan blocks are $(4)$, and the index function is $\chi_\sigma = [4;2]$. Hence $\chi_\sigma(4)=2$. Again $R_\sigma=1$, so the contribution is $q^2$.

Combining these contributions gives
\[
    A_2(q) = \frac{q^6}{(q^2 - 1)(q^4 - 1)} + \frac{q^2}{q^2 - 1} + \frac{q^3}{q^2 - 1} + q^2 + q^2.
\]
Therefore
\[
    A_2(q) = \frac{q^6}{(q^2 - 1)(q^4 - 1)} + \frac{q^2+q^3}{q^2 - 1} + 2q^2.
\]

Substituting into the coefficient formula gives
\[
\begin{aligned}
    \frac{N_2(q)}{|\Sp_4(\F_q)|} &= \binom{q}{2} \left(\frac{q^2}{q^2 - 1}+q\right)^2 + q\left(\frac{q^6}{(q^2 - 1)(q^4 - 1)} + \frac{q^2+q^3}{q^2 - 1} + 2q^2 \right) \\
    & + \frac{q^2-q}{2} \left(\frac{q^4}{q^4 - 1}+q^2\right).
\end{aligned}
\]
Simplifying gives
\[
    \frac{N_2(q)}{|\Sp_4(\F_q)|} = \frac{q^4(q^6+2q^5-q^2-2q+1)}{(q^2 - 1)(q^4 - 1)}.
\]
Since $|\Sp_4(\F_q)| = q^4(q^2 - 1)(q^4 - 1)$, we get
\[
    N_2(q) = q^8(q^6+2q^5-q^2-2q+1).
\]
Equivalently,
\[
    N_2(q) = q^{14}+2q^{13}-q^{10}-2q^9+q^8.
\]
Thus we have proved the following.

\begin{prop} \label{prop:7-1}
For $q$ a power of $2$,
    \[
    \#\{(X,Y)\in \sp_4(\F_q)^2 \mid [X,Y]=0\} = q^{14}+2q^{13}-q^{10}-2q^9+q^8.
\]
\end{prop}

\begin{remark}
    Proposition~\ref{prop:7-1} has been verified by exhaustive computation for $q=2$ and $q=4$. We represent each element of $\sp_4(\F_q)$ as
    \[
        \begin{pmatrix}
            P & Q\\
            R & P^{\mathsf T}
        \end{pmatrix},
    \]
    where $Q=Q^{\mathsf T}$ and $R=R^{\mathsf T}$. For each $X\in\sp_4(\F_q)$, the computation finds the rank of the linear map $Y\mapsto[X,Y]$ on the $10$-dimensional space $\sp_4(\F_q)$ and adds $q^{10-\operatorname{rank}(\operatorname{ad}X)}$. The resulting totals are $30{,}976$ for $q=2$ and $401{,}145{,}856$ for $q=4$, in agreement with Proposition~\ref{prop:7-1}. The SageMath implementation is supplied in Appendix~\ref{app:A}.
\end{remark}

\section*{Acknowledgments}
The work was supported by the NSF under grant DMS-2447229 as well as Jane Street. The authors gratefully acknowledge the financial support of NSF and Jane Street, and thank Texas State University for providing a great working environment and support.
Yong Yang was also partially supported by a grant from the Simons Foundation (\#918096, to YY). Liam May was also partially supported by the MIT Department of Mathematics.

\section*{Disclosure Statement}
The authors declare that they have no competing interests and no conflicts of interest.

\section*{Data Availability Statement} The SageMath \cite{sagemath} code used for the exhaustive checks in Remark~7.2 is supplied in Appendix~\ref{app:A}. No external data sets were used.

\begin{appendices}

\newpage
\section{Computational Verification}\label{app:A}

The following SageMath code verifies Proposition~\ref{prop:7-1} for
$q=2$ and $q=4$. It enumerates the elements $X\in\sp_4(\F_q)$,
computes the rank of $\operatorname{ad}X$, and sums
$q^{10-\operatorname{rank}(\operatorname{ad}X)}$.

\medskip

\begingroup
\footnotesize
\begin{verbatim}
from itertools import product

def commuting_pair_count(q):
    K = GF(q, "a")

    def sp4_element(v):
        P = matrix(K, 2, 2, v[:4])
        Q = matrix(K, 2, 2,
                   [v[4], v[5], v[5], v[6]])
        R = matrix(K, 2, 2,
                   [v[7], v[8], v[8], v[9]])
        return block_matrix([[P, Q],
                             [R, P.transpose()]])

    def coordinates(X):
        return [X[0,0], X[0,1], X[1,0], X[1,1],
                X[0,2], X[0,3], X[1,3],
                X[2,0], X[2,1], X[3,1]]

    basis = []
    for i in range(10):
        v = [K.zero()] * 10
        v[i] = K.one()
        basis.append(sp4_element(v))

    total = 0
    distribution = {}
    elements = tuple(K)

    for v in product(elements, repeat=10):
        X = sp4_element(v)
        images = [coordinates(X*Y - Y*X)
                  for Y in basis]
        rank = matrix(K, images).rank()
        dimension = 10 - rank
        distribution[dimension] = (
            distribution.get(dimension, 0) + 1
        )
        total += q**dimension

    return total, distribution

for q in [2, 4]:
    total, distribution = commuting_pair_count(q)
    formula = (q**14 + 2*q**13 - q**10
               - 2*q**9 + q**8)

    print("q =", q)
    print("ordered commuting pairs =", total)
    print("formula value =", formula)
    print("centralizer dimensions =",
          sorted(distribution.items()))

    assert total == formula
\end{verbatim}
\endgroup

\end{appendices}

\newpage
\bibliographystyle{amsalpha}
\bibliography{Bibliography}

\end{document}